\documentclass[11pt]{amsart}
\usepackage{amsmath,amssymb}	
\usepackage{graphicx}
\numberwithin{equation}{section}
\newtheorem{thm}{Theorem}[section]
\newtheorem{cor}[thm]{Corollary}
\newtheorem{lem}[thm]{Lemma}
\newtheorem{prop}[thm]{Proposition}

\newtheorem{rem}[thm]{Remark}

\def\R{\mathfrak R}

\def\F{\mathfrak F}

\def\H{\mathfrak H}

\def\R{\mathfrak R}
\def\bmu{\boldsymbol{\mu_o}}
\def\rit#1{{\mbox{\rm #1}}}
\def\modx#1#2{\equiv#1\hspace{-1mm}\mod #2}
\def\nmodx#1#2{\not\equiv#1\hspace{-1mm}\mod #2}
\def\itemx#1{\item[{\rm(#1)}]}
\def\br#1{\{#1\}}

\begin{document}
\title{Minimal equations and values of generalized lambda functions}
\author{Noburo Ishii}
\curraddr{8-155 Shinomiya-Koganezuka Yamashina-ku Kyoto\\
607-8022 Japan}
\email{Noburo.Ishii@ma2.seikyou.ne.jp} 
\begin{abstract}
 In our preceding article, we defined a generalized lambda function $\Lambda(\tau)$ and showed that $\Lambda(\tau)$ and the modular invariant function $j(\tau)$ generate the modular function field with respect to a principal congruence subgroup. In this article we shall study a minimal equation and values of $\Lambda(\tau)$. 
\end{abstract}
\maketitle
\section{Introduction}
In this article, we fix a positive integer $N$ and unless otherwise mentioned use the notation $\zeta=\exp(2\pi i/N)$ and $q=\rit{exp}(2\pi i\tau/N)$, where $\tau$ is a variable on the complex upper half plane $\H$. Let $\Gamma(N)$ be the principal congruence subgroup of level $N$, thus,
\[
\Gamma(N)=\left\{\left. \begin{pmatrix} a & b \\ c & d \end{pmatrix}\in \rit{SL}_2(\mathbb Z)~\right |~ a-1\equiv b\equiv c \equiv 0 \mod N \right\}. 
\]
In particular, $\Gamma(1)=\rit{SL}_2(\mathbb Z)$.  We denote by $A(N)$ the modular function field consisted of all modular functions  with respect to $\Gamma(N)$ having $q$-expansions with coefficients in $K_N=\mathbb Q(\zeta)$. For $\tau\in \H$, let $\wp(z;L_\tau)$ be the Weierstrass $\wp$-function relative to the lattice $L_\tau$ generated by $1$ and $\tau$.  Let $N>1$. For an element of $(r,s)$ of the group $\mathbb Z/N\mathbb Z\oplus\mathbb Z/N\mathbb Z$, we denote by $\varphi_\tau((r,s))$ the element of the group $\mathbb C/L_\tau$ congruent to $(r\tau+s)/N \mod  L_\tau$. In \cite{IN2}, for a basis $\{Q_1,Q_2\}$ of $\mathbb Z/N\mathbb Z\oplus\mathbb Z/N\mathbb Z$, we defined a generalized lambda function by
\begin{equation*}\label{genlam1}
\Lambda(\tau;Q_1,Q_2)=\frac{\wp(\varphi_\tau(Q_1);L_\tau)-\wp(\varphi_\tau(Q_1+Q_2);L_\tau)}{\wp(\varphi_\tau(Q_2);L_\tau)-\wp(\varphi_\tau(Q_1+Q_2);L_\tau)}
\end{equation*}
and showed  that $\Lambda(\tau;Q_1,Q_2)$ and the modular invariant function $j(\tau)$ generate $A(N)$ over $K_N$ if $N\ne 1,6$. For $N=2$, the function $\lambda=\Lambda(\tau;(0,1),(1,0))$ is known classically as the elliptic modular lambda function. It is proved in \cite{LS} 18.6  that  
\begin{equation*}
j=2^8\frac{(\lambda^2-\lambda+1)^3}{\lambda^2(\lambda-1)^2},
\end{equation*}
which gives a minimal equation of $\lambda$ over $\mathbb Q(j)$;
\begin{equation*}
x^6-3x^5+(6-j/256)x^4+(-7+j/128)x^3+(6-j/256)x^2-3x+1=0. 
\end{equation*}
This equation has a symmetric property concerning the coefficients, which means that the coefficient of $x^i$ is equal to that of $x^{6-i}$ for $i=0,\cdots,3$.\\
The purpose of this article is to study a minimal equation over $K_N(j)$ and values at imaginary quadratic integers for $\Lambda(\tau;Q_1,Q_2)$. In the course we obtain a similar symmetric property for coefficients of the minimal equation. In the last section, in the cases that $N=3,4$, we give equations among $\Lambda(\tau)=\Lambda(\tau;(1,0),(0,1)),j,\lambda$ and the eta quotient $\eta(\tau)/\eta(N\tau)$, and compute the values of $\Lambda(\tau)$ at imaginary quadratic integers $\alpha$ such that $\mathbb Z[\alpha]$ are maximal orders of class number one.\\
\indent The following notation will be used throughout this article.\newline 
The field $\mathbb Q(\zeta)$ and its maximal order $\mathbb Z[\zeta]$ are denoted by $K_N$ and $O_N$ respectively. For a function $f(\tau)$ and $A=\left(\begin{smallmatrix}a&b\\c&d\end{smallmatrix}\right)\in\Gamma(1)$, $f[A]_2$ and $f\circ A$ represent
\[
f[A]_2=f\left(\frac{a\tau+b}{c\tau+d}\right)(c\tau+d)^{-2}\text{ and }f\circ A=f\left(\frac{a\tau+b}{c\tau+d}\right).
\]
 The greatest common divisor of $a,b\in\mathbb Z$ is denoted by $GCD(a,b)$. For an integer $\ell$ prime to $N$, the automorphism $\sigma_\ell$ of $K_N$ is defined by $\zeta^{\sigma_\ell}=\zeta^\ell$. For an automorphism $\sigma_\ell$ and a power series $f=\sum_ma_mq^m$ with $a_m\in K_N$, $f^{\sigma_\ell}$ denotes a power series $\sum_m a_m^{\sigma_\ell}q^m$. 
\section{Auxiliary results}
In this section, we summarize the results in \cite{IN2} needed below. For the proof, see \cite{IN2} and further \cite{II2}. Let $N>2$.  For an integer $x$, let $\{x\}$ and $\mu (x)$ be the integers defined  by the following conditions:
\[
\begin{split}
&0\le \{x\}\le \frac N2,\quad \mu (x)=\pm 1,\\
&\begin{cases}\mu(x)=1\qquad &\text{if } x\modx {0,N/2}N,\\
             x\equiv \mu (x)\{x\} \mod N\qquad&\text{otherwise.}
\end{cases}
\end{split}
\]
For a pair of integers $(r,s)$ such that $(r,s)\not\equiv (0,0) \mod N$, consider a modular form of weight $2$ with respect to $\Gamma (N)$ 
$$E(\tau;r,s)=\frac 1{(2\pi i)^2}\wp((r\tau +s)/N;L_{\tau})-1/12.$$ 
\begin{prop}\label{propE} The form $E(\tau;r,s)$ has the following properties.
\begin{enumerate}
\itemx{i} $E(\tau;r+aN,s+bN)=E(\tau;r,s)$ for any integers $a$ and $b$,\\
 $E(\tau;r,s)=E(\tau;-r,-s)$. 
\itemx{ii} $E(\tau;r,s)[A]_2=E(\tau;ar+cs,br+ds)$ for $A=\left(\begin{smallmatrix}a & b \\ c & d \end{smallmatrix}\right)\in \Gamma(1)$.
\itemx{iii}
{\small
\begin{equation*}
E(\tau ;r,s)=
\begin{cases}\displaystyle
\frac{\omega}{(1-\omega)^2}+\sum_{m=1}^{\infty}\sum_{n=1}^{\infty}n(\omega^n+\omega^{-n}-2)q^{mnN}&\text{if }\{r\}=0,\\
\displaystyle\sum_{n=1}^{\infty}n u^n+\displaystyle\sum_{m=1}^{\infty}\sum_{n=1}^{\infty}n(u^n+u^{-n}-2)q^{mnN}&\text{otherwise},
\end{cases}
\end{equation*}
}
where $\omega=\zeta^{\mu (r)s}$ and $u=\omega q^{\{r\}}$.
\itemx{iv}$E(\tau;r,s)^{\sigma_\ell}=E(\tau;r,s\ell)$ for any integer $\ell$ prime to $N$.
\end{enumerate}
\end{prop}
\begin{prop}\label{prop0} Let $(r_1,s_1)$ and $(r_2,s_2)$ be pairs of integers such that\\
 $(r_1,s_1),~(r_2,s_2)\nmodx{(0,0)}N$ and $(r_1,s_1)\nmodx{\pm (r_2,s_2)}N$. Put $\omega_i=\zeta^{\mu (r_i)s_i}$. Assume that $\{r_1\}\leq \{r_2\}$.  Then
\[E(\tau;r_1,s_1)-E(\tau;r_2,s_2)=\theta q^{\{r_1\}}(1+qh(q)),\]
where $h(q)\in O_N[[q]]$ and $\theta$ is a non-zero element of $K_N$ defined as follows.
In the case of $\{r_1\}=\{r_2\}$,
\[
\theta =\begin{cases}\omega_1-\omega_2\quad&\text{if }\{r_1\}\ne 0,N/2,\\
          -\displaystyle\frac{(\omega_1-\omega_2)(1-\omega_1\omega_2)}{\omega_1\omega_2}\quad&\text{if }\{r_1\}=N/2,\\
\displaystyle\frac{(\omega_1-\omega_2)(1-\omega_1\omega_2)}{(1-\omega_1)^2(1-\omega_2)^2}\quad&\text{if }\{r_1\}=0.
\end{cases}
\]
In the case of $\{r_1\}<\{r_2\}$,
\[
\theta =\begin{cases}\displaystyle \omega_1\quad&\text{if }\{r_1\}\ne 0,\\
\displaystyle\frac{\omega_1}{(1-\omega_1)^2}\quad&\text{if }\{r_1\}=0.
\end{cases}
\]
Further $ E(\tau;r_1,s_1)-E(\tau;r_2,s_2)$ has neither zeros nor poles on $\H$.
\end{prop}
  If $Q_i=(r_i,s_i)~(i=1,2)$, then by the definition 
\begin{equation*}
\Lambda(\tau;Q_1,Q_2)=\frac{E(\tau;r_1,s_1)-E(\tau;r_1+r_2,s_1+s_2)}{E(\tau;r_2,s_2)-E(\tau;r_1+r_2,s_1+s_2)}.
\end{equation*}
Let denote $\left(\begin{smallmatrix}Q_1\\Q_2\end{smallmatrix}\right)$ a matrix $\left(\begin{smallmatrix}r_1&s_1\\r_2&s_2\end{smallmatrix}\right)\in \text{M}_2(\mathbb Z/N\mathbb Z)$.
In what follows, for an integer $k$ prime to $N$ (considered modulo $N$), the function $\Lambda(\tau;(1,0),(0,k))$ is denoted by $\Lambda_k(\tau)$ to simplify the notation. Furthermore  $\Lambda_1(\tau)$ is denoted by $\Lambda(\tau)$. 

\begin{thm}\label{prop1}
Let $\{Q_1,Q_2\}$ be a basis of the group $\mathbb Z/N\mathbb Z\oplus\mathbb Z/N\mathbb Z$. 
\begin{enumerate}

 \itemx{i} $\Lambda(\tau;Q_1,Q_2)$ has zeros and poles only at cusps of $\Gamma(N)$ and has a $q$-expansion with coefficients in $K_N$. 
\itemx{ii} $(1-\zeta)^3\Lambda(\tau;Q_1,Q_2)$ is integral over $O_N[j]$. Especially $\Lambda(\tau;Q_1,Q_2)$ is a unit of the integral closure of $ O_{N,1-\zeta}[j]$ in $A(N)$, where $O_{N,1-\zeta}$ is the localization of $O_N$ with respect to the multiplicative set $\{(1-\zeta)^n~|~n\in \mathbb Z,n\geq 0\}$. Further if $N$ is not a prime power, then $\Lambda(\tau;Q_1,Q_2)$ is a unit of the integral closure of $O_N[j]$ in $A(N)$. 
\itemx{iii} If $N\ne 6$, then $A(N)=K_N(j,\Lambda(\tau;Q_1,Q_2))$.
\itemx{iv} Let $k=\det\left(\begin{smallmatrix}Q_1\\Q_2\end{smallmatrix}\right)$ and let $A\in\Gamma(1)$ such that $\left(\begin{smallmatrix}Q_1\\Q_2\end{smallmatrix}\right)=\left(\begin{smallmatrix}1&0\\0&k\end{smallmatrix}\right)(A\mod N)$. Then
 \[\Lambda(\tau;Q_1,Q_2)=\Lambda_{k}\circ A.\]
\itemx{v}Let $k$ be an integer prime to $N$ and $A=\left(\begin{smallmatrix}a&b\\c&d\end{smallmatrix}\right)\in \Gamma(1)$. 
Then \[(\Lambda\circ A)^{\sigma_k}=\Lambda_k\circ A_k,\]
where $A_k\in\Gamma(1)$ such that $\displaystyle A_k\equiv \left(\begin{smallmatrix}a&bk\\ck^{-1}&d\end{smallmatrix}\right)\mod N$.
\end{enumerate}
\end{thm}
\begin{proof}
The assertions except the latter part of (ii) are proved in \cite{IN2}. The latter part of (ii) is deduced from the facts that $1/\Lambda(\tau,Q_1,Q_2)=\Lambda(\tau,Q_2,Q_1)$ and $(1-\zeta)$ is a unit of $O_{N,1-\zeta}$, and that $(1-\zeta)$ is a unit of $O_N$ if $N$ is not a prime power.
\end{proof}
\section{Minimal equation of $\Lambda$ over $K_N(j)$}
 Let us consider a polynomial of $X$ with coefficients in $A(1)$;
$$ \F(X)=\prod_{A\in\mathfrak R}(X-\Lambda\circ A),$$
 where $\mathfrak R$ is a transversal of the coset decomposition of $\Gamma(1)$ by $\pm\Gamma(N)$.
\begin{prop} Let $N>2$ and $N\ne 6$. 
\begin{enumerate}
\itemx i $\F(X)$ is a polynomial in $K_N[j][X]$. The degree of $\F(X)$ is equal to $d_N=\frac {N^3}2\prod_{p|N}(1-p^{-2})$, where $p$ runs over all prime divisors of $N$.
\itemx {ii} Let $\{Q_1,Q_2\}$ be a basis of $\mathbb Z/N\mathbb Z\oplus\mathbb Z/N\mathbb Z$ and $k=\det\left(\begin{smallmatrix}Q_1\\Q_2\end{smallmatrix}\right)$. Then $\F(X)^{\sigma_k}$ is a minimal equation of $\Lambda(\tau,Q_1,Q_2)$ over $K_N(j)$.
\end{enumerate}
\end{prop}
\begin{proof}
For (i), see the proof of Theorem 4.4 in \cite{IN2}. The degree $d_N$ is obviously equal to the index $[\Gamma(1):\pm\Gamma(N)]=\frac {N^3}2\prod_{p|N}(1-p^{-2})$. See 1.6 of \cite{SG}. Since $N\ne 6$, Theorem~\ref{prop1} (iii) implies that $\F(X)$ is a minimal equation of $\Lambda\circ A$ over $K_N(j)$ for any $A\in\Gamma(1)$. Therefore Theorem~\ref{prop1}~(iv) and (v) show that the equation $\F^{\sigma_k}(X)=0$ is a minimal equation of $\Lambda(\tau,Q_1,Q_2)$. 
\end{proof}
The transversal $\mathfrak R$ is adopted as follows. By \cite{Ogg}, inequivalent cusps of $\Gamma(N)$ correspond bijectively to elements of the set $\Sigma$ of pairs of integers $(a,c)$, where $a,c\in\mathbb Z/N\mathbb Z,GCD(a,c,N)=1$ and $(a,c)$ and $(-a,-c)$ are identified. 
We decompose $\Sigma$ into two disjoint subsets $\Sigma_1$ and $\Sigma_2$ where $\Sigma_1=\{\{(a,c),(c,a)\}|a\not\equiv \pm c \mod N\}$ and $\Sigma_2=\{(a,c)| a\equiv \pm c\mod N\}$. For each element $(a,c)\in\Sigma$, we take and fix a matrix $A\in \Gamma(1)$ corresponding to $(a,c)$ so that $A\equiv \left(\begin{smallmatrix}a&*\\c&*\end{smallmatrix}\right) \mod N$. If $(a,c)\in \Sigma_1$ and $A=\left(\begin{smallmatrix}x&y\\z&w\end{smallmatrix}\right)$ is the corresponding matrix to $(a,c)$, we choose $A'=\left(\begin{smallmatrix}z&-w\\x&-y\end{smallmatrix}\right)$ as a matrix corresponding to $(c,a)$. Denoting by $\mathfrak S$ the set of all corresponding matrices, with $T=\left(\begin{smallmatrix}1&1\\0&1\end{smallmatrix}\right)$,
 $$\mathfrak R=\{AT^i~|~A\in\mathfrak S,i\in\mathbb Z/N\mathbb Z\}.$$
 Further, let $\mathfrak S_1$  be the set consisted of all pairs $\{A,A'\}$ corresponding to the pairs $\{(a,c),(c,a)\}$ in $\Sigma_1$ and $\mathfrak S_2$ the set consisted of all matrices corresponding to the elements of $\Sigma_2$. 
\begin{prop}\label{prop3}Let $\nu(A)$ be the order of the $q$-expansion of $\Lambda\circ A$ for $A\in \Gamma(1)$.
\begin{enumerate}
\itemx{i} If $A\equiv\left(\begin{smallmatrix}a&*\\c&*\end{smallmatrix}\right)\mod N$, then 
$$\nu(A)=\min(\br{a},\br{a+c})-\min(\br{c},\br{a+c}).$$
\itemx{ii} Let $\{A,A'\}\in\mathfrak S_1$. Then $\nu(A)=\nu(AT^i)=-\nu(A'T^i)$ for any $i\in\mathbb Z$.
\itemx{iii} Let $A\in\mathfrak S_2$. Then $\nu(AT^i)=0$ for any $i\in\mathbb Z$.
\end{enumerate}
\end{prop}
\begin{proof}
Let $A\in \Gamma(1)$ and assume that $\displaystyle A\equiv \left(\begin{smallmatrix}a&b\\c&d\end{smallmatrix}\right) \mod N$. Since by Proposition~\ref{propE} (ii) 
$$
\Lambda\circ A=\frac{E(\tau;a,b)-E(\tau,a+c,b+d)}{E(\tau;c,d)-E(\tau,a+c,b+d)},
$$
Proposition~\ref{prop0} implies (i). Since right multiplication of a matrix with $T^i$ does not change the first column, (ii) and (iii) are immediate results of (i).
\end{proof}
For pairs of integers $(r_1,s_1)$ and $(r_2,s_2)$ such that $(r_1,s_1),(r_2,s_2)\not\equiv (0,0)$ and $~(r_1,s_1)\not\equiv \pm (r_2,s_2) \mod N$, we define 
\begin{equation*}
W(\tau;r_1,s_1,r_2,s_2)=\frac{E(\tau;r_1,s_1)-E(\tau;r_2,s_2)}{E(\tau;r_1,-s_1)-E(\tau;r_2,-s_2)}. 
\end{equation*}
The next lemma is used in the proof of Theorem~\ref{thm1} and is deduced immediately from Proposition~\ref{prop0}.
\begin{lem}\label{lem2}
Assume that $\{r_2\}\geq\{r_1\}$. Then the leading coefficient $c$ of the $q$-expansion of $W(\tau;r_1,s_1,r_2,s_2)$ is given as follows:
\begin{equation*}\label{eq1}
c=\begin{cases}-\zeta^{\mu(r_1)s_1+\mu(r_2)s_2}~&\text{ if } \{r_2\}=\{r_1\}\ne 0,N/2,\\
\zeta^{2\mu(r_1)s_1}~&\text{ if }\{r_2\}>\{r_1\}\ne 0,\\
1&\text{ otherwise}.
\end{cases}
\end{equation*}
\end{lem}

Let $F(X,Y)$ be the polynomial such that $F(X,j)=\F(X)$. Let us write $F(X,Y)$ in a polynomial of $K_N[Y][X]$;
\[F(X,Y)=\sum_{i=0}^{d_N}P_i(Y)X^{d_N-i}.\]
\begin{thm}\label{thm1} Let $N>2,\ne 6$. 
\begin{enumerate}
\itemx{i} $(1-\zeta)^{3i}P_i(Y)\in O_N[Y]$.
\itemx{ii} $P_{d_N-i}(Y)=\overline{P_i(Y)}$ for all $i$, where $\overline{P_i(Y)}$ is the complex conjugation of $P_i(Y)$.
\itemx{iii}$P_i(Y)$ has degree smaller than $i/2$. In particular $P_1(Y)$ and $P_2(Y)$ are constants.
\itemx{iv} Let $t_N$ be the number of inequivalent cusps of $\Gamma(N)$ where $\Lambda$ has poles. Then $\deg P_i(Y)\leq\deg P_{Nt_N}(Y)=[A(N):K_N(\Lambda)]$ for any $i$.
\end{enumerate}
\end{thm} 
\begin{proof}
First, we shall prove that $P_{d_N}(Y)=1$.  Proposition~\ref{prop3} implies that $P_{d_N}(Y)$ has no poles. Therefore, $P_{d_N}(Y)$ is a constant and equals to the product of leading coefficients of $q$-expansions of $\Lambda\circ A~(A\in\mathfrak R)$, since $d_N$ is even. Let us consider a partial product; 
$$\prod_{(A,A')\in\mathfrak S_1}\prod_{i=0}^{N-1}\Lambda\circ(AT^i)\Lambda\circ(A'T^{-i}).$$

Let $(A,A')\in\mathfrak S_1$. If $A\equiv \left(\begin{smallmatrix}a&b\\c&d\end{smallmatrix}\right) \mod N$, then 
\begin{align*}
\Lambda\circ(AT^i)\cdot\Lambda\circ(A'T^{-i})=W(\tau;a,ia+b,a+c,i(a+c)+b+d)\\
\times W(\tau;c,-(ic+d),a+c,-(i(a+c)+b+d)).
\end{align*}
Let $\delta$ and $\delta'$ be the leading coefficient of 
$\prod_{i=0}^{N-1}W(\tau;a,ia+b,a+c,i(a+c)+b+d)
$
and 
$
\prod_{i=0}^{N-1}W(\tau;c,-(ic+d),a+c,-(i(a+c)+b+d))
$
respectively.
Let $\{a\}=\{a+c\}\ne 0,N/2$. Then by Lemma~\ref{lem2},
\begin{align*}
\delta=&(-1)^N\zeta^{\mu(a)\Sigma_i(ai+b)+\mu(a+c)\Sigma_i(a+c)i+b+d}=(-1)^N\zeta^{\Sigma_i(\mu(a)a+\mu(a+c)(a+c))i}\\
&=(-1)^N\zeta^{(\br a+\br{a+c})N(N-1)/2}=(-1)^N.
\end{align*}
 In other cases, we have easily that $\delta=1$. Similarly, $\delta'=(-1)^N $ if $\br{c}=\br{a+c}$ and $\delta'=1$ otherwise. Since $\br{a}=\br{a+c}=\br{c}$ does not hold, the number of $A$ with $\br{a}=\br{a+c}$ or $\br{c}=\br{a+c}$ is equal to $2\varphi(N)$, where $\varphi(x)$ is Euler totient function. Therefore the leading coefficient of the partial product is $1$. If $A\in\mathfrak S_2$, then $c\equiv a,-a\mod N$. Therefore $\br{a}\ne 0,N/2$ and $\br{a+c}=0\text{ or }\br{2a}$, and if $N\ne 3$, then $\br{a}\ne\br{a+c}$. Proposition~\ref{prop0} gives immediately that the leading coefficient of $\prod_{i=0}^{N-1}\Lambda\circ(AT^i)$ is $1$. Hence $P_{d_N}(Y)=1$. For $N=3$, a direct calculation in \S 6 (I) gives that $P_{d_3}(Y)=1$. Next we shall prove the assertion (ii). By Theorem~\ref{prop1}, $\Lambda^{-1}=\Lambda_{-1}\circ S$,where $S=\left(\begin{smallmatrix}0&1\\-1&0\end{smallmatrix}\right)$. Therefore, $X^{d_N}F(1/X,j)=0$ is a minimal equation of $\Lambda_{-1}$. Further by Theorem~\ref{prop1}~(v), since $\Lambda_{-1}=\Lambda^{\sigma_{-1}}$, $F(X,j)^{\sigma_{-1}}=0$ is a minimal equation of $\Lambda_{-1}$. This shows that $X^{d_N}F(1/X,j)=\overline{F(X,j)}$, and $P_{d_N-i}(Y)=\overline{P_i(Y)}$ for all $i$. The assertion (iii) is deduced from the fact that $\nu(A)>-N/2$ for all $A\in \Gamma(1)$ and the $q$-expansion of $j$ has order $-N$. The assertion (i) is a consequence of  Theorem~\ref{prop1} (ii). For (iv), see Theorem \ref{thm3}~(i) and (ii) in the next section.
\end{proof}

\section{Minimal equation of $j$ over $K_N(\Lambda)$}
In this section,  we shall study the polynomial $F(X,Y)$ of $Y$ with coefficients in $K_N[X]$.  Because the following result can be proved by the almost same argument in section 3 of \cite{II}, we outline the proof. Let $X(N)$ be the modular curve associated with $\Gamma(N)$ defined over $K_N$.

\begin{thm}\label{thm3} Let $N>2,\ne 6$. Put 
\begin{equation*}
F(X,Y)=Q_0(X)Y^{\ell_N}+Q_1(X)Y^{\ell_N-1}+\cdots +Q_{\ell_N}(X).
\end{equation*}
\begin{enumerate}
\itemx{i} $F(\Lambda,Y)$ is a minimal equation of $j$ over $K_N(\Lambda)$. The degree $\ell_N$ of $F(X,Y)$ with respect to $Y$ is equal to
$[A(N):K_N(\Lambda)]$.
\itemx{ii}Let $t_N$ be the number of cusps of $X(N)$ where $\Lambda$ has poles. Let $\alpha_k (k=1,\cdots,d_N/N-2t_N)$ be cusps of $X(N)$ where $\Lambda$ has neither poles nor zeros. Then $\Lambda(\alpha_k)\in K_N$ and $Q_0(X)=c(X^{t_N}\prod_k(X-\Lambda(\alpha_k))^N$ for a non-zero constant $c\in K_N$.
\itemx{iii}$Q_{\ell_N}(X)=F(X,0)=c_1H_1(X)^3$,where $H_1(X)=\prod_{k=1}^{d_N/3}(X-\Lambda(\rho_k))$ and $\rho_k$ are points on $X(N)$ lying over $\rho=(1+\sqrt{-3})/2$, and $c_1$ is a non-zero constant. Further $H_1(0)\ne 0$.
\itemx{iv} $F(X,1728)=c_2H_2(X)^2$, where $H_2(X)=\prod_{k=1}^{d_N/2}(X-\Lambda(\tau_i))$ and $\tau_k$ are points on $X(N)$ lying over $i$, and $c_2$ is a non-zero constant. Further $H_2(0)\ne 0$.
\itemx{v} For any complex number $c\ne 0,1728$, $F(X,c)=0$ has no multiple roots.
\end{enumerate}
\end{thm}
\begin{proof}
For (i), see the argument in the latter part of Lemma 3 of \cite{II}. The ramification points of $X(N)/X(1)$ are $i\infty,\rho=(1+\sqrt{-3})/2$ and $i=\sqrt{-1}$ and the ramification index of these points are $N,3$ and $2$ respectively. See 1.6 of \cite{SG}. We know that $j(\rho)=0,j(i)=1728$. We note that zero points of $\Lambda$ on $X(N)$ are only cusps and the number of zero points of $\Lambda$ on $X(N)$ is $t_N$ by Proposition~\ref{prop3}. The remaining assertions are obtained from these facts by using the argument in section 3 of \cite{II}. 
\end{proof}
The integers $\ell_N$ and $t_N$ are expressed in the following sums defined for a positive integer $M$, a positive divisor $L$ of $M$ and a non-negative integer $k$;
\begin{equation}\label{sum}
I_k(L,M)=\sum _t t^k~\text{and }J_k(L,M)=\sum_{u}u^k,
 \end{equation}
 where $t$ runs over all integers such that $0<t<M/2,GCD(t,L)=1$ and $u$ runs over integers such that $0<u<M/2,GCD(u,L)=1$ and $u\modx {-M}3$. 
\begin{prop}\label{prop4} Let $N>2,~\ne 6$. Then
\begin{equation*}
\begin{split}
&\ell_N=I_1(N,N)+2\sum_{0<a<N/3}I_1(GCD(a,N),N-3a)+\delta_1(N),\\
&t_N=I_0(N,N)+2\sum_{0<a<N/3}I_0(GCD(a,N),N-3a)+\delta_0(N),
\end{split}
\end{equation*}
where $a$ runs over all positive integers smaller than $N/3$ and
\begin{equation*}
\begin{split}
&\delta_1(N)=\begin{cases}3(I_1(N/3,N/3)-J_1(N/3,N/3))~~&\text{if }N\modx 03,\\
J_1(N,N) &\text{if } N\not\equiv 0 \mod 3,\end{cases}\\
&\delta_0(N)=\begin{cases}I_0(N/3,N/3)-J_0(N/3,N/3)~~&\text{if }N\modx 03,\\
J_0(N,N) &\text{if } N\not\equiv 0 \mod 3.\end{cases}\\
\end{split}
\end{equation*}
\end{prop}
\begin{proof}  By (i) of Theorem~\ref{thm3} ,
$$\ell_N=[A(N):K_N(\Lambda)]=-\sum_{\nu(A)<0,A\in\mathfrak S}\nu(A).$$ 
Let $A\in \mathfrak S$ correspond to $(a,c)\in\Sigma$ with $\nu(A)<0$. Then $\nu(A)=\br{a}-\min(\br{c},\br{a+c})<0$. Assuming that $0\leq a<N/2$ and $-N/2< c\leq N/2$, we shall determine $(a,c)\in\Sigma$ satisfying the condition:
\begin{equation}\label{pole}
a<\min(\br{c},\br{a+c}).
\end{equation}
 For each $a$, to determine $c$, put $k=\br c-a$. Since $a<\br c\leq N/2$, $0<k\leq (N-2a)/2$. If $a=0$, then 
\begin{equation}\label{eq0}
0<k<N/2,~GCD(k,N)=1,~\nu(A)=-k.
\end{equation}
 Let $a>0$. If $-N/2<c<0$, then $c=-(a+k)$. Since $\br{a+c}=k< a+k=\br{c}<N/2$, by \eqref{pole} $a<k<(N-2a)/2$ and $\nu(A)=a-k$. By substituting $t$ for $k-a$, we have
\begin{equation}\label{eqa}
0<t<\frac{N-4a}2,~GCD(a,t,N)=1,~\nu(A)=-t.
\end{equation}
Let $0\leq c\leq N/2$. Then $c=a+k$. 
If $k\leq (N-3a)/2$, then $\br c\leq\br{a+c}$ and $\nu(A)=a-\br c=-k$. If $k>(N-3a)/2$, then $\br c>\br{a+c}$. Therefore, by \eqref{pole}, $k<N-3a$ and $\nu(A)=a-\br{a+c}=(3a+k)-N$. Putting them together,
\begin{equation}\label{eqb}
0<k\leq \frac{N-3a}2,~GCD(a,k,N)=1,~\nu(A)=-k,
\end{equation}
\begin{equation}\label{eqd}
\begin{cases}
&\frac{N-3a}2<k\leq \min(N-3a-1,\frac{N-2a}2),GCD(a,k,N)=1,\\
&\nu(A)=3a+k-N.
\end{cases}
\end{equation}
 In \eqref{eqd}, by substituting $t$  for  $N-3a-k$, 
\begin{equation}\label{eqe}
\max(1,\frac{N-4a}2)\leq t<\frac{N-3a}2,~GCD(a,t,N)=1,~\nu(A)=-t.
\end{equation}
Therefore bringing together \eqref{eqa} and \eqref{eqe},
\begin{equation}\label{eqf}
0<t<\frac{N-3a}2,~GCD(a,t,N)=1,\nu(A)=-t.
\end{equation}
 Hence by \eqref{eq0},~\eqref{eqb},~\eqref{eqf},
\begin{equation*}
\ell_N=I_1(N,N)+2\sum_{0<a<N/3}I_1(GCD(a,N),N-3a)+\delta_1(N),
\end{equation*}
where $\delta_1(N)=\sum_a(N-3a)/2$, and $a$ runs over all integers such that $0<a<N/3, a\equiv N \mod 2$ and $GCD(a,(N-a)/2)=1$. Let $t=(N-3a)/2$. If $N\not\equiv 0\mod 3$, then $\delta(N)=\sum_t t$, where $0<t<N/2,GCD(t,N)=1,t\equiv -N \mod 3$. Therefore $\delta(N)=J_1(N,N)$. If $N\equiv 0\mod 3$, putting $t=3s$,then $\delta(N)=3\sum_s s$, where $0<s<N/6,GCD(s,N/3)=1,s\not\equiv -N/3\mod 3$. Therefore $\delta_1(N)=3(I_1(N/3,N/3)-J_1(N/3,N/3))$.
Since $t_N$ is the number of $(a,c)\in \Sigma$ with $\nu(A)<0$, the above argument also gives the result for $t_N$.
\end{proof}

In general, it is not easy to express $\ell_N$ explicitly. However in the case that $N$ is a prime power, we have the following results.
\begin{prop}\label{prop5} Assume that $N$ is a power of a prime number.
\begin{enumerate}
\itemx{i} Let $N=2^m$, where $m$ is an integer greater than $1$. Then
$$t_{N}=\frac{3\cdot2^{2m-3}-2^{m-1}-(-1)^m}3,~\ell_N=\frac{2^{3m-4}-(-1)^m}3. $$
\itemx{ii} Let $N=3^m$, where $m$ is a positive integer. If $m=1$, then $\ell_N=t_N=1$. If $m\geq 2$, then
$$t_N=4\cdot 3^{2m-3}-2\cdot 3^{m-2},~\ell_{N}=2\cdot 3^{3m-4}.$$
\itemx{iii} Let $N=p^m$, where $p$ is a prime number greater than $3$ and $m$ is a positive integer. Then
\begin{align*}
&t_N=\frac{p^{2m}-p^{2m-2}-2p^m+2p^{m-1}}6+\begin{cases}0 &\text{ if }p\modx13,\\
      \frac 13 &\text{ if }p\modx 23,m\text{:odd},\\ -\frac13 &\text{ if }p\modx 23,m\text{:even},
\end{cases}\\
&\ell_N=\frac{p^{3m}-p^{3m-2}}{36}+\begin{cases}\frac{p-1}9 &\text{ if }p\modx13,\\
\frac{p+1}9 &\text{ if }p\modx 23,m\text{:odd},\\
-\frac{p+1}9 &\text{ if }p\modx 23,m\text{:even}.
\end{cases}
\end{align*}
\end{enumerate}
\end{prop}
\begin{proof}
Let $N=p^m$ be a prime power. Since $GCD(a,k,N)\ne 1$ if and only if $p$ divides $a$ and $b$ at the same time, 
\begin{equation*}
\sum_{0<a<N/3}I(a,N)=\sum_{0<a<\frac N3}~\sum_{0<k<\frac{N-3a}2}k-p\left(\sum_{0<a<\frac N{3p}}~~\sum_{0<k<\frac{(N/p)-3a}2}k\right).
\end{equation*}
 The assertions can be obtained from this observation by easy but tedious calculation. We omit details.
\end{proof}
Proposition~\ref{prop4} implies that $\ell_N=1$ only for $N=3,4$. On the other hand, the genus of $X(N)$ is $0$ for $N=1,2,3,4$ and $5$. For $N=3,4$, $\Lambda$ is a generator of $A(N)$ over $K_N$ and $j$ can be expressed in a rational function of $\Lambda$, thus, $j=-Q_1(\Lambda)/Q_0(\Lambda)$. Since $\ell_5=4$, $\Lambda$ is not a generator for $N=5$. 
\begin{rem}
 $A(1) ~(\text{resp}.A(2), A(5))$ is generated by $j~ (\text{resp}.\lambda, X_2)$, where $X_2$ is a product of Klein forms defined in \cite{II}.
\end{rem}
See Appendix for the sums $I_k(L,M)$ and $J_k(L,M)~(k=0,1)$. 

\section{Values of $\Lambda$}
In this section, we summarize the results concerning the values of $\Lambda$.
\begin{thm}\label{thmv} Let $\alpha\in\H$.
\begin{enumerate}
\itemx{i}If $\Lambda(\alpha)$ is algebraic, then $\alpha$ is imaginary quadratic  or transcendental.
\itemx{ii} If $\alpha$ is algebraic but is not imaginary quadratic, then $\Lambda(\alpha)$ is transcendental.
\itemx{iii} Let $\alpha$ be imaginary quadratic. Then $(1-\zeta)^3\Lambda(\alpha)$ and $(1-\zeta)^3/\Lambda(\alpha)$ are algebraic integers. 
\itemx{iv} Let $\alpha$ be an imaginary quadratic and assume that $\mathbb Z[\alpha]$ is the maximal order of the field $K=\mathbb Q(\alpha)$. Then $K(\zeta,j(\alpha),\Lambda(\alpha))$ is the ray class field of $K$ modulo $N$ if $N\ne 6$.
\itemx{v} $\Lambda(\alpha)^{-1}=\overline{\Lambda(\alpha/|\alpha|^2)}$. If the absolute value of $\alpha$ is $1$, then the absolute value of $\Lambda(\alpha)$ is $1$.
\end{enumerate}
\end{thm}
\begin{proof}
Since $\Lambda(\alpha)$ is algebraic if and only if $j(\alpha)$ is algebraic, the assertions (i) and (ii) are deduced from Schneider's result on the values of $j$ \cite{SC}. The assertion (iii) follows from Theorem~\ref{prop1}~(ii) and the fact that $j(\alpha)$ is an algebraic integer (\cite{Cox},Theorem 10.23, \cite{SG}, Theorem 4.14).  The assertion (iv) is the result in Theorem 4.5 \cite{IN2}. Since $\Lambda_{-1}=(1/\Lambda)\circ S$,
\begin{equation*}
\overline{\Lambda(\alpha)}=\Lambda_{-1}(-\overline{\alpha})=(1/\Lambda)(S(-\overline{\alpha}))=(1/\Lambda)(1/\overline{\alpha})=1/\Lambda(\alpha/|\alpha|^2).
\end{equation*}
Hence if $|\alpha|=1$, then $\Lambda(\alpha)\overline{\Lambda(\alpha)}=1$. This shows (vi).
\end{proof}
\begin{cor}
 Let $\alpha$ be imaginary quadratic. If $N$ is not a prime power, then $\Lambda(\alpha)$ is a unit. If $N$ is a power of a prime $p$, then up to a sign the norm $N_{\mathbb Q(\Lambda(\alpha))/\mathbb Q}(\Lambda(\alpha))$ is a power of $p$.
\end{cor}
\begin{proof}
This is a direct result of Theorem \ref{thmv} (iii).
\end{proof}
\begin{rem} If $N$ is a prime power, then the value of $\Lambda$ is not necessarily integral at an imaginary quadratic number. For example, for $N=3$, $\Lambda(\frac{1+\sqrt{-11}}2)$ has norm $3^{-3}$ and for $N=4$, $\Lambda(\sqrt{-2})$ has norm $1/2$. See \rm{Example (I), (II)} in the last section.
 \end{rem}
\begin{prop}\label{propdeg} Let $K$ be the imaginary quadratic number field of discriminant $D_K$ and $H$ the Hilbert class field of $K$. Let $N=\prod p_i^{e_i}$ be the prime decomposition of $N$. Put $K_i=\mathbb Q(\zeta_{p_i^{e_i}})$. Assume that $K\ne \mathbb Q(\sqrt{-1}),~\mathbb Q(\sqrt{-3})$.
\begin{enumerate}
\itemx i If $K$ is contained in one of $K_i$, then 
$$[\R_N:H(\zeta)]=N\prod_i\left(1-\left(\frac{D_K}{p_i}\right)p_i^{-1}\right).$$
\itemx {ii} If $K$ is not contained in any $K_i$, then 
\begin{equation*}
[\R_N:H(\zeta)]=\begin{cases}2^sN\prod_i\left(1-\left(\frac{D_K}{p_i}\right)p_i^{-1}\right)&\text{  if  }K\subset K_N,\\
2^{s-1}N\prod_i\left(1-\left(\frac{D_K}{p_i}\right)p_i^{-1}\right)&\text{  otherwise}.
\end{cases}
\end{equation*}
\end{enumerate}
Here the integer $s$ is defined as follows. Denoting by $r$  the number of odd prime factors of $GCD(D_K,N)$,
$$
s=\begin{cases} r+1 &\text{ if }N\modx 48, D_K\modx 48,\\
                r+1&\text{ if }N\modx 08,D_K\modx 02,\\
                r &\text{ otherwise.}
\end{cases}
$$
\end{prop}
\begin{proof} Let $h_K$ be the class number of $K$. If $K\subset K_i$, then $H\cap K(\zeta)$=K. Therefore $[H(\zeta):K]=\varphi(N)h_K/2$. Let $K\not\subset K_i$ for  any $i$. Let $L$ be the composite field of quadratic extensions $K\left(\sqrt{(-1)^{\frac{p-1}2}p}\right)$ for all odd prime factors $p$ of $GCD(D_K,N)$. By using elementary algebraic number theory, in the cases that $N\modx 48,~D_K\modx 48$ and that $N\modx 08,~D_K\modx 04$, $H\cap K(\zeta)=L(\sqrt{m})$, where $m$ is an integer suitably chosen from $\{-2,-1,2\}$. In other cases, $H\cap K(\zeta)=L$. Therefore  $[H\cap K(\zeta):K]=2^s$. This shows that $[H(\zeta):K]= \varphi(N)h_K/2^{s+1}$ or $\varphi(N)h_K/2^{s}$ according to $K\subset \mathbb Q(\zeta)$ or not. Since 
$$[\R_N:K]=\frac{h_k\varphi(N)}2N\prod_{p|N}\left(1-\left(\frac{D_K}p \right)p^{-1}\right),
$$ we have our assertion.
\end{proof}
\begin{cor} Let $\alpha$ be an imaginary quadratic number and assume that $\mathbb Z[\alpha]$ is the maximal order of the field $K=\mathbb Q(\alpha)$. If $K\ne \mathbb Q(\sqrt{-1}),~\mathbb Q(\sqrt{-3})$, then the degree of a minimal equation of $\Lambda(\alpha)$ over $K(j(\alpha),\zeta)$ is equal to $[\R_N:H(\zeta)]$ and it is given in Proposition~\ref{propdeg}.
\end{cor}
\section{Examples in the cases $N=3$ and $4$}
Let $\eta(\tau)$ be the Dedekind eta function. Let $g_N(\tau)=\eta(\tau)/\eta(N\tau)$. By \cite{newman}, $g_N(\tau)^{\frac{24}{N-1}}$ is a generator of the modular function field of the Hecke group $\Gamma_0(N)$ for $N=3,4$.  Let $X_0(N)$ be the modular curve associated with $\Gamma_0(N)$. In this section, for $N=3,4$, we give some equations among $j,\lambda,g_N$ and $\Lambda$, and compute values of $\Lambda$ for some imaginary quadratic integers $\alpha$ such that $\mathbb Z[\alpha]$ is the maximal order of class number one. These values $\Lambda(\alpha)$ can be obtained from comparing their approximate values with solutions of $F(X,j(\alpha))=0$. The values of $j$ for such $\alpha$ are listed in \cite{Cox} \S 12 C. \newline

(I) the case $N=3$. We have $d_3=12$ and $\zeta=\rho-1$.
\begin{equation*}
j=-3^4\sqrt{-3}\left(\frac{(\Lambda+\zeta-1)(\Lambda+(\zeta-1)/3)(\Lambda+\zeta+1)(\Lambda-\zeta-1)}{\Lambda(\Lambda-1)(\Lambda+\zeta)}\right)^3.
\end{equation*}
\begin{align*}
 F(X,1728)=(X-i+\zeta)^2&(X+i+\zeta)^2(X-1+i\zeta)^2(X-1-i\zeta)^2\\
&\times (X-i-i\zeta)^2(X+i+i\zeta)^2.
\end{align*}
Since $g_3^{12}\in A(3)$, $g_3^{12}$ is a rational function of $\Lambda$, thus,
$$g_3^{12}=81(1-\zeta)\frac{(\Lambda-1)(\Lambda+\zeta)}{\Lambda^3}.$$
Since $X(3)$ is totally ramified at the point $\rho/(\rho+1)$ of $X_0(3)$ of ramification index $3$, this gives the following equation;
$$(\Lambda g_3^4)^3+(3\Lambda)^3=(3(\Lambda+\zeta-1))^3.
$$
The solutions of the equation $X^3+Y^3=Z^3$ are given  by using modular functions $\Lambda$ and $g_3$.
The values of $\Lambda$ at $\alpha=i,\rho$ and $\sqrt{-2}$ are;
$$\Lambda(i)=i\rho,~\Lambda(\rho)=-\rho,~\Lambda(\sqrt{-2})=(\sqrt{-3}-\sqrt{-2})\rho.$$
 We remark that the ray class field of $\mathbb Q(\sqrt{-m})$ of conductor $3$ is a quadratic extension of $\mathbb Q(\sqrt{-m})$ if $m\equiv -1 \mod 3$ and it is a cyclic extension of degree $4$ otherwise. The values $v(m)$ of $\Lambda$ at $\alpha=(1+\sqrt{-m})/2$, where $m=3,7,11,19,43,67$ and $163$, are as follows. \newline
The case $m\equiv -1 \mod 3$.
$$v(11)=\sqrt{-3}(-1+2\sqrt{-11}-3\sqrt{-3})/18.$$
The case $m\equiv 1 \mod 3$.
$$v(m)=\frac{1+\Omega-\beta\sqrt{3\sqrt{-3}\Omega}}2,$$
where
$$
(\beta,\Omega)=\begin{cases}(1,(\sqrt{-3}-\sqrt{-7})/2)&\text{ if }m=7,\\
(2,5\sqrt{-3}-2\sqrt{-19})&\text{ if }m=19,\\
(6,53\sqrt{-3}-14\sqrt{-43})&\text{ if }m=43,\\
(14,293\sqrt{-3}-62\sqrt{-67})&\text{ if }m=67,\\
(154,35573\sqrt{-3}-4826\sqrt{-163})&\text{ if }m=163.
\end{cases}
$$
Those values of $\Lambda$ are units except $v(11)$. The norm of $v(11)$ is $3^{-3}$.\newline
(II) the case $N=4$. We have $d_4=24$.
{\small
\begin{equation*}
j=-2^6\frac{((\Lambda^2-(1-2i)\Lambda-i)(\Lambda^2-(2-i)\Lambda-i)(\Lambda^2-\Lambda+1)(\Lambda^2+i\Lambda-1))^3}{(\Lambda(\Lambda+i)(\Lambda-1)(\Lambda-1+i)(\Lambda-\frac{1-i}2))^4}.
\end{equation*}
}
\begin{align*}
F(X,&1728)=(X^2-2(1-i)X-i)^2(X^2+i)^2(X^2-(1-i)X-\frac{1+i}2)^2\\
&\times(X^2-(1-i)X+\frac{1-i}2)^2(X^2+2iX-1-i)^2(X^2-2X+1-i)^2.
\end{align*} 
Since $\lambda,g_4^8\in A(4)$, $\lambda$ and $g_4^8$ are rational functions of $\Lambda$, thus,
$$\lambda=2i\left(\frac{\Lambda-(1-i)/2}{\Lambda(\Lambda-1+i)}\right)^2,$$
$$g_4^8=-2^6(1-i)\frac{(\Lambda+i)(\Lambda-1)(\Lambda+(i-1)/2)}{\Lambda^4}.$$
Since $X(4)$ is totally ramified at the cusp $1/2$ of $X_0(4)$ of ramification index $4$, the following equation is deduced;
$$(\Lambda g_4^2)^4+(2\Lambda)^4=(2(\Lambda+1-i))^4.
$$
The values $\Lambda(\alpha)$ at $\alpha=i,\rho,\sqrt{-2}$ and $(1+\sqrt{-7})/2$ are;
\begin{align*}
&\Lambda(i)=(i-1)/\sqrt{-2},~\Lambda(\rho)=i(\rho-1),\\
&\Lambda(\sqrt{-2})=(1-i)\left(1-\sqrt{1+\sqrt{2}}\right)/2,\\
&\Lambda((1+\sqrt{-7})/2)=(1-3i+(1+i)\sqrt{-7})/2.
\end{align*}
For $m=11,19,43,67$ and $163$, the ray class field of $\mathbb Q(\sqrt{-m})$ of conductor $4$ is an abelian extension of degree $3$  over $\mathbb Q(\sqrt{-m},i)$ and the value $\Lambda((1+\sqrt{-m})/2)$ is a root of the following equation $EQ(m)$ of degree $3$. \\

The case $m\equiv 11 \mod 16$.
{\small
\begin{align*}
&EQ(11):\\
&X^3-\frac{3+2i-\sqrt{-11}}2X^2+\frac{7+2i+(2i-1)\sqrt{-11})}2X\\
&\phantom{aaaaaaaaaaaaaaaaaaaaaaaaaaaaaaaaaa}-\frac{3(1-i)+(1+i)\sqrt{-11}}2=0,\\
&EQ(43):\\
&X^3-\frac{3+58i-9\sqrt{-43}}2X^2+\frac{119+58i+9(2i-1)\sqrt{-43})}2X\\
&\phantom{aaaaaaaaaaaaaaaaaaaaaaaaaaaaaaaaaa}-\frac{59(1-i)+9(1+i)\sqrt{-43}}2=0
\end{align*}
}
The case $m\equiv 3 \mod 16$.\\

$EQ(m):X^3+\Omega iX^2+\overline{\Omega}X+i=0$,\\

where
$$
\Omega=\begin{cases}\frac{(3-8i-3\sqrt{-19})}2 &\text{ if }m=19,\\
\frac{(3-216i-27\sqrt{-67})}2 &\text{ if } m=67,\\
\frac{(3-8000i-627\sqrt{-163})}2 &\text{ if }m=163.
\end{cases}
$$
Those values are units except $\Lambda(\sqrt{-2})$ and $\Lambda(\frac{1+\sqrt{-7}}2)$. The norm of $\Lambda(\sqrt{-2})$ (resp. $\Lambda(\frac{1+\sqrt{-7}}2)$) is $1/2$ (resp. $2^3$).
\appendix
\section{Sums $I_k(L,M)$ and $J_k(L,M)$}
Let $M$ be a positive integer and $L$ be a positive divisor of $M$. Let $I_k(L,M)$ and $J_k(L,M)$ be the sums defined by \eqref{sum}. We shall compute these sums for $k=0,1$.   For a positive integer $n$, we denote by $n^*$ the product of all prime divisors of $n$. Put
$S_k(n)=\sum_{i=1}^ni^k$. If $L^*>2$, then by a routine argument, we have
\begin{equation}\label{sum1}
I_k(L,M)=\sum_{d|L^*}\bmu(d)d^kS_k(\lfloor M/2d\rfloor_0),
\end{equation}
where $\bmu(x)$ is the M\"{o}bius function and $d$ runs over all positive divisors of $L^*$, and $\lfloor x\rfloor_0$ denotes the greatest integer less than $x$.
 It is easy to see that $I_1(L,M)=0$ for $M=1,2$ and that 
\begin{equation*}
I_1(1,M)=\begin{cases}(M^2-1)/8~~&\text{if $M$ is odd,}\\
 M(M-2)/8 &\text{otherwise}. 
\end{cases}
\end{equation*}
Further if $L^*=2$, then 
\begin{equation*}
I_1(L,M)=\begin{cases}\left(\frac M4\right)^2~~&\text{if } M\equiv 0\mod 4,\\ 
\left(\frac{M-2}4\right)^2 &\text{if } M\equiv 2 \mod 4.
\end{cases}
\end{equation*}
For the remaining cases we have
\begin{prop}\label{prop_sum}  Let $L^*>2$. If $\ell$ is the number of prime divisors of $L$,then
\begin{equation*}
I_1(L,M)=\varphi(L^*)(M^2/L^*-(-1)^\ell\epsilon)/8,
\end{equation*}
where
\begin{equation*}
 \epsilon=\begin{cases} 1~~&\text{if } M\equiv 1\mod 2,\\
                                       2 &\text{if } M\equiv 2\mod 4, L^*\equiv 0\mod 2,\\
0 &\text{otherwise.}
\end{cases}
 \end{equation*}
\end{prop}
\begin{proof}
By \eqref{sum1}, we see that
\begin{equation*}
I_1(L,M)=\sum_{d|L^*}\bmu(d)d\left(\frac{\lfloor M/2d\rfloor_0\left(\lfloor M/2d\rfloor_0+1\right)}2\right).
\end{equation*}
Since the other cases are treated similarly, we give the proof only in the case that $M\equiv 2\mod 4$ and $L^*\equiv 0\mod 2$. By dividing the sum into two partial sums for odd divisors and for even divisors, 
\begin{equation*}
{\small 
 I_1(L,M)=\frac12\sum_{d|L^*/2}\bmu(d)d\left(\lfloor M/2d\rfloor_0\left(\lfloor M/2d\rfloor_0+1\right)-2\lfloor M/4d\rfloor_0\left(\lfloor M/4d\rfloor_0+1\right)\right)}.
\end{equation*}
 Since $\lfloor M/2d\rfloor_0=M/2d$ and $\lfloor M/4d\rfloor_0=(M/2d-1)/2$,
\begin{equation*}
I_1(L,M)=\sum_{d|L^*/2}\bmu(d)(M^2/(16d)+M/4+d/4).
\end{equation*}
Since $L^*/2>1$, we have $\sum_{d|L^*/2}\bmu(d)=0$ and
\begin{equation*}
\begin{split}
&\sum_{d|L^*/2}\bmu(d)/d=\prod_{p|L^*/2}(1-1/p)=2\varphi(L^*)/L^*,\\
&\sum_{d|L^*/2}\bmu(d)d=\prod_{p|L^*/2}(1-p)=-(-1)^\ell\varphi(L^*),
\end{split}
\end{equation*}
where $p$ runs over all prime divisors of $L^*/2$.
Therefore,
\begin{equation*}
I_1(L,M)=\varphi(L^*)(M^2/L^*-2(-1)^\ell)/8.
\end{equation*}
\end{proof}
By \eqref{sum1}, we have immediately
\begin{prop}\label{prop_sum0}
\begin{equation}
I_0(L,M)=\begin{cases}(M-2)/2~~&\text{if }L^*=1,M\equiv 0\mod 2,\\
(M-1)/2~~&\text{if }L^*=1,M\equiv 1\mod 2,\\
(M-2)/4~~&\text{if }L^*=2,M\equiv 2\mod 4,\\
M\varphi(L^*)/2L^* &\text{otherwise}.
\end{cases}
\end{equation}
\end{prop}
We remark that $I_0(L,M)$ is the number of positive integers smaller than $M/2$ prime to $L$.
\begin{prop}\label{J1} Let $M\not\equiv 0\mod 3$ and $L^*>2$. Then
\begin{equation*}
J_1(L,M)=\varphi(L^*)(M^2/L^*+(-1)^\ell(8-9\epsilon))/24,
\end{equation*}
where 
\begin{equation*}
\epsilon=\begin{cases}1~~~&\text{if }M\equiv 1 \mod 2,\\
               2   &\text{if }M\equiv 2 \mod 4,L^*\equiv 0\mod 2,\\
               0   &\text{otherwise}.
\end{cases}
\end{equation*}
\end{prop}
\begin{proof}In \eqref{sum}, put $u=3s''-N$. Then the condition on $u$ is equivalent to the condition on $s''$ that $N/3<s''<N/2, GCD(s'',N)=1$. Therefore
\begin{equation}\label{sum2}
\begin{split}
J_1(L,M)=&\sum_{s'}(3s'-M)-\sum_s(3s-M)\\
&=3I_1(L,M)-MI_0(L,M)-\sum_s(3s-M),
\end{split}
\end{equation}
where $0<s'<M/2,GCD(s',L)=1$ and $0<s<M/3,GCD(s,L)=1$. For the sum for $s$,
\begin{equation*}
\sum_s(3s-M)=\sum_{d|L^*}\bmu(d)\lfloor M/3d\rfloor_0\left(3d\left(\lfloor M/3d\rfloor_0+1\right)/2-M\right).
\end{equation*}
Since $\lfloor M/3d\rfloor_0\left(3d\left(\lfloor M/3d\rfloor_0+1\right)/2-M\right)=(3M-M^2/d-2d)/6$ for any $d$, 
\begin{equation*}
\begin{split}
\sum_s(3s-M)&=\left(\sum_{d|L^*}\bmu(d)(3M-M^2/d-2d)\right)/6\\
&=-\varphi(L^*)(M^2/L^*+2(-1)^\ell)/6.
\end{split}
\end{equation*}
By Propositions \ref{prop_sum0},~\ref{prop_sum} and \eqref{sum2}, we have our assertion.
\end{proof}
\begin{rem} If $L^*=2$, we have
\begin{equation*}
J_1(L,M)=\begin{cases}(M^2-12M+20)/48~~&\text{if }M\equiv 4 \mod 4,\\
 (M^2-16)/48 &\text{if }M\equiv 0 \mod 4.
\end{cases}
\end{equation*}
\end{rem}
\begin{prop} Let $M\not\equiv 0\mod 3$ and $L^*>2$. Let $\ell$ be the number of prime factors of $L^*$. Then
\begin{equation*}
J_0(L,M)=\frac 16\left(M\varphi(L^*)/L^*-\left(\frac{M}3\right)2^{\ell}\epsilon\right),
\end{equation*}
where
\begin{equation*}
\epsilon=\begin{cases}1~~&\text{if $L^*$ has no prime factors congruent to $1\mod 3$},\\
                      0 &\text{otherwise},
\end{cases}
\end{equation*}
and $\left(\frac *3\right)$ is the Legendre symbol.
\end{prop}
\begin{proof} By the same argument in Proposition~\ref{J1},
\begin{equation*}
J_0(L,M)=I_0(L,M)-\sum_{d|L^*}\bmu(d)\lfloor M/3d\rfloor_0.
\end{equation*} 
Since $\lfloor M/3d\rfloor_0=(M-d)/3d$ (resp. $(M-2d)/3d$) if $d\equiv M \mod 3$ (resp. $d\equiv -M \mod 3$),
\begin{equation*}
\sum_{d|L^*}\bmu(d)\lfloor M/3d\rfloor_0=\frac 13\left(M\varphi(L^*)/L^*+\left(\frac{M}3\right)\sum_{d'}\bmu(d')\right),
\end{equation*}
where $d'$ runs over all divisors of $L^*$ congruent to $1\mod 3$. It is easy to see that
\begin{equation*}
\sum_{d'}\bmu(d')=2^{\ell_2-1}\sum_{d''|L_1*}\bmu(d''),
\end{equation*}
where $L_1^*$ is the product of all prime factors of $L^*$ congruent to $1\mod 3$ and $\ell_2$ be the number of prime factors of $L^*$ congruent to $-1\mod 3$.
Therefore by Proposition~\ref{prop_sum0}, we have our assertion.
\end{proof}
It is obvious that $J_1(L,M)=J_0(L,M)=0$ if $M,L\equiv 0\mod 3$. 

\vspace{5mm}
\end{document}